\documentclass[11pt]{article}

\usepackage[top=1in, left=1.50in, bottom=1in, right=1.50in]{geometry}

\usepackage{hyperref}  
\usepackage{setspace} 

\usepackage{graphicx}
\usepackage{amsfonts}
\usepackage{amssymb}
\usepackage{amsmath}
\usepackage{amsthm}

\newtheorem{theorem}{Theorem}[section]

\newtheorem{lemma}[theorem]{Lemma}
\newtheorem{conjecture}[theorem]{Conjecture}
\newtheorem{claim}{}[theorem]

\newcommand{\del}{\backslash}

\title{Graphs with girth $2\ell+1$ and without longer odd holes are $3$-colorable}

\author{Rong Chen\\
\\
Center for Discrete Mathematics,\ \ Fuzhou University\\
Fuzhou,\ \ P. R. China}

\begin{document}

\maketitle

\footnote{

Email: rongchen@fzu.edu.cn (R. Chen).
}

\begin{abstract}
For a number $\ell\geq 2$, let $\mathcal{G}_{\ell}$ denote the family of graphs which have girth $2\ell+1$ and have no odd hole with length greater than $2\ell+1$. Plummer and Zha conjectured that every 3-connected and internally 4-connected graph in $\mathcal{G}_2$ is 3-colorable. Wu, Xu, and Xu conjectured that every graph in $\bigcup_{\ell\geq2}\mathcal{G}_{\ell}$ is 3-colorable. Chudnovsky et al. and Wu et al., respectively, proved that every graph in $\mathcal{G}_2$ and $\mathcal{G}_3$ is 3-colorable. In this paper, we prove that every graph in $\bigcup_{\ell\geq5}\mathcal{G}_{\ell}$ is 3-colorable.

\end{abstract}

\section{Introduction}

All graphs considered in this paper are finite, simple, and undirected.
A {\em proper coloring} of a graph $G$ is an assignment of colors to the vertices of $G$ such that no two adjacent vertices receive the same color. A graph is {\em $k$-colorable} if it has a proper coloring using at most $k$ colors. The {\em chromatic number} of $G$, denoted by $\chi(G)$, is the minimum number $k$ such that $G$ is $k$-colorable.

The {\em girth} of a graph $G$, denoted by $g(G)$, is the minimum length of a cycle in $G$. A {\em hole} in a graph is an induced cycle of length at least four. An {\em odd hole} means a hole of odd length. For any integer $\ell\geq2$, let $\mathcal{G}_{\ell}$ be the family of graphs that have girth $2\ell + 1$ and have no odd holes of length at least $2\ell + 3$.  Robertson conjectured in \cite{ND11} that the Petersen graph is  the only graph in $\mathcal{G}_2$ that is 3-connected and internally 4-connected.
Plummer and Zha \cite{PM14} disproved Robertson's conjecture and proposed the conjecture that all 3-connected and internally 4-connected graphs in $\mathcal{G}_2$ have bounded chromatic number, and the strong conjecture that such graphs are 3-colorable. The first was proved by Xu, Yu, and Zha \cite{XB17}, who proved that all graphs in $\mathcal{G}_2$ is 4-colorable. Chudnovsky and Seymour \cite{MC22} proved that every graph in $\mathcal{G}_2$ is 3-colorable. In the same paper, Chudnovsky and Seymour also gave a short and pretty proof of the result in \cite{XB17}.
Wu, Xu, and Xu \cite{WD2204} showed that graphs in $\bigcup_{\ell\geq2}\mathcal{G}_{\ell}$  are 4-colorable and proposed the following conjecture.

\begin{conjecture}\label{conj}(\cite{WD2204}, Conjecture 6.1.)
Every graph in $\bigcup_{\ell\geq2}\mathcal{G}_{\ell}$ is $3$-colorable.
\end{conjecture}

We say that a graph $G$ has an {\em odd $K_4$-subdivision} if some subgraph $H$ of $G$ is isomorphic to a $K_4$-subdivision and whose face cycles are all odd holes of $G$. The subgraph $H$ is also induced by (\cite{CZ22}, Lemma 2.7). In the same paper, the author and Zhou proved

\begin{theorem}\label{ex-odd K4}(\cite{CZ22}, Theorem 1.1.)
For each graph $G$ in $\bigcup_{\ell\geq5}\mathcal{G}_{\ell}$, if $G$ has an odd $K_4$-subdivision, then $\chi(G)=3$.
\end{theorem}

Based on Theorem \ref{ex-odd K4}, in this paper we prove that Conjecture \ref{conj} holds for $\ell\geq5$. That is,

\begin{theorem}\label{main thm}
All graphs in $\bigcup_{\ell\geq5}\mathcal{G}_{\ell}$ are $3$-colorable.
\end{theorem}


We conjecture
\begin{conjecture}\label{conj2}
For a graph $G$ without triangles, if all odd holes of $G$ have the same length, then $G$  is $3$-colorable.
\end{conjecture}

\begin{conjecture}\label{conj2}
For an integer $\ell\geq2$ and a graph $G$ with $g(G)=2\ell+1$, if the set of lengths of odd holes of $G$ have $k$ members, then $G$ is $(k+2)$-colorable.
\end{conjecture}

\section{Preliminary}
A {\em cycle} is a connected $2$-regular graph. 
Let $G$ be a graph. For any subset $U\subseteq V(G)$, let $G[U]$ be the induced subgraph of $G$ defined on $U$.
For subgraphs $H$ and $H'$ of $G$, set $|H|:=|E(H)|$ and $H\Delta H':=E(H)\Delta E(H')$.
Let $H\cup H'$ denote the subgraph of $G$ with vertex set $V(H)\cup V(H')$ and edge set $E(H)\cup E(H')$. Let $H\cap H'$ denote the subgraph of $G$ with edge set $E(H)\cap E(H')$ and without isolated vertex. Let $N(H)$ be the set of vertices in $V(G)-V(H)$ that have a neighbour in $V(H)$. Set $N[H]:=N(H)\cup V(H)$.

Let $P$ be an $(x,y)$-path and $Q$ a $(y,z)$-path.
When $P$ and $Q$ are internally disjoint paths, 
let $PQ$ denote the $(x,z)$-path $P\cup Q$. Evidently, $PQ$ is a path when $x\neq z$, and $PQ$ is a cycle when $x=z$.
Let $P^*$ denote the set of internal vertices of $P$. When $u,v\in V(P)$, let $P(u,v)$ denote the subpath of $P$ with ends $u, v$. For simplicity, we will let $P^*(u,v)$ denote $(P(u,v))^*$.

For a number $k\geq2$, a graph $G$ is {\em $k$-vertex-critical} if $\chi(G)=k$ and $\chi(G-v)\leq k-1$ for each vertex $v$ of $G$.

\begin{lemma}\label{2-edge-cut}(\cite{CZ22}, Lemma 2.2.)
For any number $k\geq4$, each $k$-vertex-critical graph has no $2$-edge-cut.
\end{lemma}

For an $i$-vertex path $P$ of a connected graph $G$, if $G-V(P)$ is disconnected, then we say that $P$ is a {\em $P_i$-cut}. Usually, a $P_2$-cut is also called a {\em $K_2$-cut}. Evidently, every $k$-vertex-critical graph has no $K_2$-cut. Chudnovsky and Seymour in \cite{MC22} proved that every $4$-vertex-critical graph $G$ in $\mathcal{G}_2$ has no $P_3$-cut. In fact, their methods can also be used to prove the following result.

\begin{lemma}\label{P3}(\cite{MC22})
For any number $\ell\geq2$, every $4$-vertex-critical graph in $\mathcal{G}_{\ell}$ has neither $K_2$-cut nor $P_3$-cut.
\end{lemma}

A {\em theta graph} is a graph that consists of a pair of distinct vertices joined by three internally disjoint paths. Let $C$ be a hole of a graph $G$. A path $P$ of $G$ is a {\em chordal path} of $C$ if $C\cup P$ is an induced theta-subgraph of $G$. Since $g(G)=2\ell+1$ and each odd hole of $G$ has length $2\ell+1$, by simple computation we have the following result, which will be frequently used.

\begin{lemma}\label{easy case}(\cite{CZ22}, Lemma 2.3.)
Let $C$ be an odd hole of a graph $G\in \mathcal{G}_{\ell}$. Let $P$ be a chordal path of $C$, and $P_1,P_2$ be the internally disjoint paths of $C$ that have the same ends as $P$. Assume that $|P|$ and $|P_1|$ have the same parity. Then $$|P_1|=1\ \text{or}\ \ell\geq|P_2|< |P_1|=|P|\geq \ell+1.$$ In particular, when $|P_1|\geq2$, all chordal paths of $C$ with the same ends as $P_1$ have length $|P_1|$.
\end{lemma}

Let $C_1,C_2$ be odd holes of a graph $G\in \mathcal{G}_{\ell}$ such that $C_1\cup C_2$ is a theta subgraph $G$. When $C_1\cup C_2$ is not induced, since $|C_1\Delta C_2|\leq4\ell$ and $g(G)=2\ell+1$, we have that $|C_1\Delta C_2|=4\ell$ and $G[V(C_1\cup C_2)]$ is an odd $K_4$-subdivision. For the similar reason, if $C$ is an even hole of length $4\ell$  of $G$, then $C$ has at most two chords, and when $C$ has two chords, $G[V(C)]$ is an odd $K_4$-subdivision. 

\begin{figure}[htbp]
\begin{center}
\includegraphics[height=6cm]{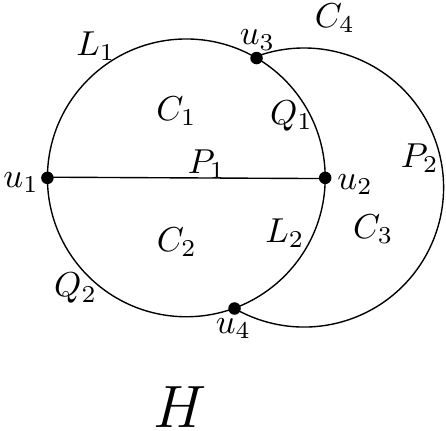}
\caption{$u_1,u_2,u_3,u_4$ are the degree-3 vertices of $H$. The face cycles $C_1, C_2$ are odd holes, and $C_3, C_4$ are even holes. $\{P_1,P_2\}$, $\{Q_1,Q_2\}$, $\{L_1,L_2\}$ are the pairs of vertex disjoint arrises of $H$.}
\label{H}
\end{center}
\end{figure}

\section{Graphs has a subgraph isomorphic to a $K_4$-subdivision}
Let $H$ be a graph that is isomorphic to a $K_4$-subdivision. For a path $P$ of $H$ whose ends are degree-3 vertices of $H$, if $P$ contains exactly two degree-3 vertices of $H$, then we say it is an {\sl arris} of $H$. Evidently, $H$ has exactly six arrises. We say that a graph $G$ contins a {\em balanced $K_4$-subdivision} if $G$ has an induced subgraph $H$ isomorphic to a $K_4$-subdivision and exactly two face cycles of $H$ are odd holes of $G$.

\begin{lemma}\label{odd k4}
If a graph $G$ in $\mathcal{G}_{\ell}$ contains a balanced $K_4$-subdivision $H$ that is pictured as the graph in Figure \ref{H}, then the following hold.
\begin{itemize}
 \item[(1)] $|P_1|\leq\ell$ and $1\in \{|Q_1|, |Q_2|, |L_1|, |L_2|\}$.
\item[(2)] $|P_2|\geq\ell$. 
\item[(3)] Assume that $G$ has no odd $K_4$-subdivision. If $|Q_1|=1$ and $|L_2|\geq2$, then no vertex $v\in V(G)-V(H)$ has two neighbours in $V(H)$. 
\end{itemize}
\end{lemma}
\begin{proof}
Since $C_1, C_2$ are odd holes and $g(G)=2\ell+1$, we have $|P_1|\leq\ell$.
Assume that $1\notin \{|Q_1|, |Q_2|, |L_1|, |L_2|\}$. Since $P_2Q_2$ is a chordal path of $C_1$ and $C_4$ is an even hole, we have $|P_2|+|Q_2|=|L_1|$ by Lemma \ref{easy case}. Similarly, we have $|P_2|+|L_1|=|Q_2|$, which is a contradiction. So (1) holds.

Now we consider (2). By (1) and symmetry we may assume that $|Q_1|=1$. Since $|P_1|\leq\ell$ by (1), we have  $|L_1|\geq\ell$, so $Q_1P_2$ is a chordal path of $C_2$. Then $|Q_1P_2|\geq\ell+1$ by Lemma \ref{easy case}. So $|P_2|\geq\ell$. This proves (2).

Next, we prove that (3) is true. Assume not. Let $v\in V(G)-V(H)$ have at least two neighbours in $V(H)$. Since $|L_2|>1$ and $C_3$ is an even hole, $C_2\Delta C_3$ is an odd hole and $|L_2|\geq\ell+1$ by Lemma \ref{easy case}. Moreover, since a vertex not in an odd hole has at most one neighbour in the odd hole, each arris of $H$ has at most one vertex adjacent to $v$.

\begin{claim}\label{2neighbour}
The vertex $v$ has exactly one neighbour in $V(C_1\cup C_2)$ and exactly one neighbour in $V(P^*_2)$.
\end{claim}
\begin{proof}[Subproof.]
Since $v$ has at most one neighbour in $V(P_2)$, it suffices to show that $v$ has at most one neighbour in $V(C_1\cup C_2)$. Assume not. Then $v$ has exactly two neighbours in $V(C_1\cup C_2)$ with one in $V(C_1)-V(P_1)$ and the other in $V(C_2)-V(P_2)$. Then $G[V(C_1\cup C_2)\cup\{v\}]$ is a balanced $K_4$-subdivision as $G$ has no odd $K_4$-subdivision,  which is not possible by (2).
\end{proof}

Note that $C_2\Delta C_3$ is an odd hole. Since $v$ can not have two neighbours in an odd hole of $H$, the vertex $v$ has no neighbour in $V(P_1)$. For the same reason, if $v$ has a neighbour in $V(Q_1\cup Q_2)$, the neighbours of $v$ in $V(H)$ must be in $V(C_1\cup C_2)$, which is not possible by \ref{2neighbour}. Hence, $N(v)\cap V(C_1\cup C_2)\subseteq V(L^*_1\cup L^*_2)$. Assume that $v$ has a neighbour in $V(L^*_1)$. Since $|C_4|\leq 4\ell-2$ and $g(G)=2\ell+1$, by \ref{2neighbour}, the two cycles in $G[V(C_4)\cup\{v\}]$ containing $v$ are odd holes. Moreover, since $|L_2|\geq\ell+1$, the subgraph $G[V(C_4)\cup\{v\}]\cup P_1\cup Q_1$ is an odd  $K_4$-subdivision, which is not possible. Similarly we can show that $v$ can not have a neighbour in $V(L^*_2)$. This proves (3). 
\end{proof}

Let $H$ be an induced subgraph of $G$ pictured as the graph in Figure \ref{H}. When $|Q_1|=1$, we say that $H$ is a {\em balanced $K_4$-subdivision of type $(1,2)$} if $|L_2|>1$. Since $|Q_1|=1$ implies that $|L_1|>1$, this definition is well defined.

Let $H_1,H_2$ be vertex disjoint induced subgraphs of a graph $G$.  An induced $(v_1,v_2)$-path $P$ is a {\em direct connection} linking $H_1$ and $H_2$ if $v_1$ is the only vertex in $V(P)$ having a neighbour in $H_1$ and $v_2$ is the only vertex in $V(P)$ having a neighbour in $H_2$. Evidently, the set of internal vertices of each shortest induced path joining $H_1$ and $H_2$ induces a direct connection linking $H_1$ and $H_2$.

\begin{theorem}\label{bal k4(1,2)}
Let $\ell\geq5$  be an integer and $G$ be a graph in $\mathcal{G}_{\ell}$. If $G$ is $4$-vertex critical,  then $G$ does not contain a balanced $K_4$-subdivision of type $(1,2)$.
\end{theorem}
\begin{proof}
By Theorem \ref{ex-odd K4}, $G$ does not have an odd $K_4$-subdivision. Assume to the contrary that $G$ has a balanced $K_4$-subdivision $H$ of type $(1,2)$ that is pictured as the graph in Figure \ref{H}. Without loss of generality we may assume that $H$ is chosen with $|H|$ as small as possible.
By Lemma \ref{odd k4} (1) and symmetry we may assume that $|Q_1|=1$. By Lemmas \ref{odd k4} (2) and \ref{easy case}, we have
\[|Q_1|=1,\  |P_2|\geq\ell,\ |L_1|, |L_2|\geq\ell+1,\ \text{and}\ |P_1|<\ell. \eqno{(4.1)}\]

Let $e,f$ be the edges in $P_2$ incident with $u_3,u_4$, respectively.
Let $P$ be a direct connection in $G\del\{e,f\}$ linking $P^*_2$ and $H-V(P^*_2)$. Lemma \ref{2-edge-cut} implies that such $P$ exists. Let $v_1,v_2$ be the ends of $P$ with $v_2$ having a neighbour in $V(P^*_2)$ and $v_1$ having a neighbour in $V(H)-V(P^*_2)$. By Lemma \ref{odd k4} (3), both $v_1$ and $v_2$ have a unique neighbour in $V(H)$. Let $x$ be the neighbour of $v_1$ in $V(H)$ and $y$ the neighbour of $v_2$ in $V(H)$. Set $P':=xv_1Pv_2y$. Then $H\cup P'$ is an induced subgraph of $G$. By Lemma \ref{odd k4} (3) again, we have $|P'|\geq3$.

\begin{claim}\label{x=u2}
$x\neq u_2$.
\end{claim}
\begin{proof}[Subproof.]
Assume that $x=u_2$. Set $C'_3:=L_2P_2(u_4,y)P'$. Since $|L_2|\geq\ell+1$, we have that $C'_3$ is an even hole by Lemma \ref{easy case}. Since $|P'|\geq3$ and $C_2\Delta C'_3$ is an odd hole, $(H\cup P')-V(L^*_2)$ is a balanced $K_4$-subdivision of type $(1,2)$ whose  edge number is less than $|H|$, which is a contradiction to the choice of $H$.
\end{proof}

\begin{claim}\label{location 1}
$x\notin P_1$.
\end{claim}
\begin{proof}[Subproof.]
Assume not. By \ref{x=u2}, we have $x\in V(P_1)-\{u_2\}$. Set $C'_1:=P_1(x,u_2)Q_1P_2(u_3,y)P'$. Since $|P_1|<\ell$ by (4.1), it follows from Lemma \ref{easy case} that $C'_1$ is an odd hole. Then $P_1(x,u_1)Q_2P_2(u_4,y)P'$ is an even hole. Moreover, since $|L_2|\geq\ell+1$, we have $x=u_1$ and $|Q_2|=1$ by Lemma \ref{easy case}. Since $C_1$ and $C'_1$ are odd holes, we have $|L_1|>|P'|$. Moreover, since $|P'|\geq3$, the subgraph $C'_1\cup C_2\cup P_2$ is a balanced $K_4$-subdivision of type $(1,2)$ whose  edge number is less than $|H|$, which is a contradiction.
\end{proof}

\begin{claim}\label{x=u3}
When $x=u_3$, we have $|P_2(u_3,y)|=1$ or $|P_2(u_3,y)|\geq\ell+1$.
\end{claim}
\begin{proof}[Subproof.]
Assume that $|P_2(u_3,y)|\neq1$. Since $Q_1P_2$ and $Q_1P'P_2(y,u_4)$ are chordal paths of $C_2$ that have the same ends as $L_2$, they have length $|L_2|$ by Lemma \ref{easy case}. So $|P_2(u_3,y)P'|=2|P_2(u_3,y)|$, implying that $|P_2(u_3,y)|\geq\ell+1$.
\end{proof}

\begin{claim}\label{location of x+}
If $x\in V(L_1^*)$, then $xu_3$, $yu_3\in E(G)$, and $|Q_2|=1$.
\end{claim}
\begin{proof}[Subproof.]
Set $C'_3:=L_1(x,u_3)P_2(u_3,y)P'$. When $C'_3$ is an odd hole, since $C'_3\Delta C_4$ is an odd hole, $C_1\cup C'_3\cup P_2\cup Q_2$ is an odd $K_4$-subdivision, which is not possible. So $C'_3$ is an even hole.
Assume that $xu_3\notin E(G)$. Then $|C'_3|=2|L_1(x,u_3)|$ by Lemma \ref{easy case}. Moreover, since $C_1\Delta C'_3$ is an odd hole, $(H\cup P')-V(L_1^*(x,u_3))$ is a balanced $K_4$-subdivision of type $(1,2)$ whose edge number is less than $|H|$, which is a contradiction to the choice of $H$. So $xu_3\in E(G)$.

Assume that $yu_3\notin E(P_2)$. Since $Q_1P_2$ and $Q_1 L_1(u_3,x)P'P_2(y,u_4)$ are chordal paths of $C_2$ that have the same ends as $L_2$, they have length $|L_2|$ by Lemma \ref{easy case}. Moreover, since $xu_3\in E(G)$ and $|L_1|\geq\ell+1$ imply $|L_1(x,u_1)|\geq2$, the subgraph $(H\cup P')-V(P_2^*(y,u_3))$ is a balanced $K_4$-subdivision of type $(1,2)$ whose edge number is less than $|H|$, which is a contradiction. So $yu_3\in E(P_2)$, implying that $|P'|\geq2\ell$. Since $C'_3\Delta C_3\Delta C_1$ is an odd cycle of length at least $2\ell+3$, we have $|Q_2|=1$. This proves \ref{location of x+}.
\end{proof}

\begin{claim}\label{x in Q2}
$x\notin V(Q^*_2)$.
\end{claim}
\begin{proof}[Subproof.]
Assume not. 
Set $C'_3:=Q_2(x,u_4)P_2(u_4,y)P'$. Assume that $C'_3$ is an odd hole. Since $C_1\cup C_2\cup (C_3\Delta C'_3)$ is an odd $K_4$-subdivision when $yu_4\notin E(P_2)$, we have $yu_4\in E(P_2)$. Since $|Q_2|<\ell$ as $|L_2|\geq\ell+1$, the graph $C_4\Delta C'_3$ is an odd hole with length larger than $3\ell$ by (4.1), which is not possible. So $C'_3$ is an even hole. Since $|Q_2|<\ell$, it follows from Lemma \ref{easy case} that $xu_4\in E(Q_2)$. Moreover, since $P'$ is a chordal path of the odd hole $C_2\Delta C_3$ and $C'_3$ is an even hole, $\ell+1\leq |P'|=|P_2(y,u_4)|+1\leq|P_2|<|L_1|$ by Lemma \ref{easy case}. Hence, $C_2\cup C_3\cup P'$ is a balanced $K_4$-subdivision of type $(1,2)$ whose edge number is less than $|H|$, which is a contradiction.
\end{proof}

Let $u'_4$ be the neighbour of $u_4$ in $L_2$.
\begin{claim}\label{location 2}
If $x\in V(L_2)$, then $x\in \{u_4,u'_4\}$. In particular, when $x=u'_4$, we have $yu_4\in E(P_2)$; and when $x=u_4$, we have $yu_4\in E(P_2)$ or $|P_2(y,u_4)|\geq\ell+1$.
\end{claim}
\begin{proof}[Subproof.]
Set $C'_3:=L_2(x,u_4)P_2(u_4,y)P'$. Assume that $x=u_4$ and $yu_4\notin E(P_2)$. Since $Q_1P_2$ and $Q_1P_2(u_3,y)P'$ are chordal paths of $C_2$ with the same ends, $|C'_3|=2|P_2(y,u_4)|$ by Lemma \ref{easy case}, so $|P_2(y,u_4)|\geq\ell+1$. That is, when $x=u_4$, \ref{location 2} holds. So we may assume that $x\neq u_4$.

By \ref{location 1}, we may assume that $x\in V(L^*_2)$. When $C'_3$ is an odd hole, since $x\notin \{u_2,u_4\}$, the graph $C_2\cup C_3\cup P'$ is an odd $K_4$-subdivision. So $C'_3$ is an even hole. When $x\neq u'_4$, by Lemma \ref{easy case}, we have that $xu_2\in E(H)$, $|C'_3|=2|L_2(x,u_4)|$ and $C_3\Delta C'_3$ is an odd hole, so $(H\cup P')-V(L_2^*(x,u_4))$ is a balanced $K_4$-subdivision of type $(1,2)$ whose edge number is less than $|H|$, which is not possible. Hence, $x=u'_4$. When $yu_4\notin E(P_2)$, since $Q_1P_2(u_3,y)P'$ is a chordal path of $C_2$, we have $|C'_3|=2|P_2(y,u_4)|$ by Lemma \ref{easy case}. So $(H\cup P')-V(P_2^*(y,u_4))$ is a balanced $K_4$-subdivision  of type $(1,2)$ whose edge number is less than $|H|$, which is not possible. So $yu_4\in E(G)$.
\end{proof}

Let $v$ be a vertex in $V(P^*_2)$.  When $|P_2|\geq \ell+1$, let $|P_2(u_3,v)|=\ell-1$, otherwise let $|P_2(u_3,v)|=\ell-2\geq3$. The vertex $v$ is well defined as $|P_2|\geq\ell$. Since $|P_2|\leq 2\ell-2$, we have $2\leq |P_2(u_4,v)|\leq \ell-1$. Let $u,w\in N(v)\cap V(P^*_2)$ with $u$ closer to $u_4$ and $w$ closer to $u_3$ along $P_2$. Since $\{uv,vw\}$ is not an edge cut of $G$ by Lemma \ref{2-edge-cut}, there is a shortest induced path $Q$ in $G\del\{uv,vw\}$ linking $v$ and $H-v$. By Lemma \ref{odd k4} (3), $H\cup Q$ is an induced subgraph of $G$.  Let $v'$ be the other end of $Q$ that is not equal to $v$. When $v'\in V(C_1\cup C_2)$, since $Q^*$ is a direct connection in $G\del\{e,f\}$ linking $P_2^*$ and $H-V(P^*_2)$, we get a contradiction to \ref{location 1}-\ref{location 2}. So $v'\in V(P^*_2)-\{v\}$. Assume that $v'\in V(P_2^*(u_3,v))$. When $v'\neq w$, either $Q_2P_1Q_1P_2(u_3,v')QP_2(v,u_4)$ is an odd hole of length at least $2\ell+3$ or $QP_2(v,v')$ is a cycle of length $2|P_2(v,v')|$ that is at most $2\ell-2$ by the choice of $v$, which is not possible. So $v'=w$. By symmetry, we have $v'\in \{u,w\}$.

Now we show that $\{v,v'\}$ is a $K_2$-cut of $G$, implying that Theorem \ref{bal k4(1,2)} holds from Lemma \ref{P3}. Assume not. Let $R'$ be a direct connection in $G-\{v,v'\}$ linking $Q^*$ and $H-\{v,v'\}$. Let $s_1,s_2$ be the ends of $R'$ with $s_1$ having a neighbour in $V(H)-\{v,v'\}$ and $s_2$ having a neighbour in $V(Q^*)$. By Lemma \ref{odd k4} (3), $s_1$ has a unique neighbour, say $t_1$, in $V(H)-\{v,v'\}$. Since $Q$ is chosen with $|Q|$ as small as possible, $s_2$ has a unique neighbour, say $t_2$, in $V(Q^*)$. Set $R:=t_1s_1R's_2t_2$. Then $H\cup Q\cup R$ is an induced subgraph of $G$ or some vertex in $\{v,v'\}$ has a neighbour in $V(R')$. When $t_1\in V(P_2)$, let $P'_2$ be an induced $(u_3,u_4)$-path in $P_2\cup Q\cup R$ that is not equal to $P_2$. Since $Q_1P'_2$ is a chordal path of $C_2$ with length $|L_2|$ by Lemma \ref{easy case}, there is a cycle in $P_2\cup P'_2$ with length at most $2\ell$ by the choice of $v$, which is not possible. So $t_1\in V(H)-V(P_2)$. Then $R\cup Q^*$ contains a direct connection in $G\del\{e,f\}$ linking $P_2^*$ and $H-V(P_2^*)$. Since $|P_2(u_3,v)|\geq\ell-2\geq3$ and $2\leq |P_2(u_4,v)|\leq \ell-1$, by \ref{location 1}-\ref{location 2}, we have that $v'=u$, $t_1=u'_4$, and $uu_4\in E(P_2)$. Let $P'_2$ be an induced $(u_3,u'_4)$-path in $(P_2-\{u_4\})\cup Q\cup R$. Since $Q_1P'_2$ and $Q_1P_2$ are chordal paths of $C_2$, by Lemma \ref{easy case}, we have $|P'_2(u'_4,v)|=|P_2(u_4,v)|-1$, so $u'_4v\in E(G)$, which is not possible.
\end{proof}

\section{Jumps over an odd hole}
Let $C$ be an odd hole of a graph $G$ and $s,t\in V(C)$ nonadjacent. Let $P$ be an induced $(s,t)$-path. If $V(C)\cap V(P^*)=\emptyset$, we call $P$ a {\em jump} or an {\em $(s,t)$-jump} over $C$. Let $Q_1,Q_2$ be the internally disjoint $(s,t)$-paths of $C$.
If some vertex in $V(Q^*_1)$ has a neighbour in $V(P^*)$ and no vertex in $V(Q^*_2)$ has a neighbour in $V(P^*)$, we say that $P$ is a {\em local jump over $C$ across $Q^*_1$}. When there is no need to strengthen $Q^*_1$, we will also say that $P$ is a {\em local jump over $C$}. In particular, when $|V(Q^*_1)|=1$, we say that $P$ is a {\em local jump over $C$ across one vertex}. When no vertex in $V(Q^*_1\cup Q^*_2)$ has a neighbour in $V(P^*)$, we say that $P$ is a {\em short jump over $C$}. Hence, short jumps over $C$ are chordal paths of $C$. But chordal paths over $C$ maybe not short jumps over $C$ as the ends of a chordal path maybe adjacent.
When $P$ is a short jump over $C$, if $PQ_1$ is an odd hole, we say that $PQ_1$ is a {\em jump hole} over $C$ and $P$ is a {\em short jump over $C$ across $Q^*_1$}. Note that our definition of local jumps over $C$ is a little different from the definition given in \cite{MC22}. By our definitions, no short jump  over $C$ is local.

\begin{lemma}\label{jump parity}
Let $C$ be an odd hole of a graph $G\in \mathcal{G}_{\ell}$. Let $P,Q_1,Q_2$ be defined as the last paragraph. If $P$ is a local or short jump over $C$ across $Q^*_1$, then $|P|,|Q_2|$ have the same parity, that is, $PQ_2$ is an even hole and $PQ_1$ is an odd cycle.
\end{lemma}
\begin{proof}
When $P$ is short, the lemma holds from its definition. So we may assume that $P$ is local. Since some vertex in $V(Q^*_1)$ has a neighbour in $V(P^*)$ and $g(G)=2\ell+1$, we have $|P|\geq 2\ell$. So $PQ_2$ is an even hole. This proves Lemma \ref{jump parity}.
\end{proof}

Let $C$ be a cycle of a graph $G$ and $e,f$ be chords of $C$. If no cycle in $C\cup e$ containing $e$ contains the two ends of $f$, we say $e,f$ are {\em crossing}; otherwise they are {\em uncrossing}. Note that by our definition, $e,f$ are uncrossing when they share an end. When $C$ is an odd hole and $P_i$ is an $(u_i,v_i)$-jump over $C$ for each integer $1\leq i\leq2$, if $u_1v_1$ and $u_2v_2$ are uncrossing chords of $C$ after we replace $P_1$ by the edge $u_1v_1$ and $P_1$ by the edge $u_2v_2$, we say that $P_1,P_2$ are {\em uncrossing}; otherwise, they are {\em crossing}.

\begin{lemma}\label{jump}
Let $C$ be an odd hole of a graph $G\in \mathcal{G}_{\ell}$. If a jump $P$ over $C$ is not local, then $G[V(C\cup P)]$ contains a short jump over $C$.
\end{lemma}
\begin{proof}
Without loss of generality we may assume that $P$ is not a short jump over $C$.
Let $s,t$ be the ends of $P$ and $Q_1,Q_2$ be the internally disjoint $(s,t)$-paths of $C$. Since $P$ is neither local nor short, there are $u_1,u_2\in V(P^*)$ such that $u_1$ has a neighbour in $V(Q_1^*)$ and $u_2$ has a neighbour in $V(Q_2^*)$, and such that no vertex in $V(P^*(u_1,u_2))$ has a neighbour in $V(C)$. Since no vertex outside an odd hole has two neighbours in the odd hole, $u_1\neq u_2$ and $u_i$ has a unique neighbour, say $w_i$, in $V(C)$ for each integer $1\leq i\leq2$. Then $w_1u_1P(u_1,u_2)u_2w_2$ is a short jump over $C$. This proves Lemma \ref{jump}.
\end{proof}

\begin{lemma}\label{jump-local}
Let $C$ be an odd hole of a graph $G\in \mathcal{G}_{\ell}$. If $P$ is a local $(v_1,v_2)$-jump over $C$, then $G[V(C\cup P)]$ contains a local jump over $C$ across one vertex or a short jump over $C$.
\end{lemma}
\begin{proof}
Assume not. Among all local jumps over $C$ not satisfying Lemma \ref{jump-local}, let $P$ be chosen with $|P|$ as small as possible. Let $u_1,u_2\in V(P^*)$ have a neighbour in $V(C)-\{v_1,v_2\}$ that are closest to $v_1,v_2$, respectively. Let $w_1,w_2$ be the unique neighbour of $u_1,u_2$ in $V(C)$, respectively. Since $w_1u_1P(u_1,v_1)$ is a short jump over $C$ when $w_1v_1\notin E(G)$, we have $w_1v_1\in E(G)$. Similarly, we have $w_2v_2\in E(G)$. Hence, when $u_1=u_2$, $w_1=w_2$ and $w_1$ is adjacent to $v_1,v_2$, which is a contradiction.  So $u_1\neq u_2$. Since $P$ is not a local jump over $C$ across one vertex, we have $w_1\neq w_2$. Let $u_3$ be the neighbour of $w_1$ in $V(P^*)$ closest to $v_2$. Note that $u_3$ maybe equal to $u_1$. By the choice of $u_2$, we have $u_2\in V(P(u_3,v_2))$. Set $P':=w_1u_3P(u_3,v_2)$. Since $P'$ is a local jump over $C$ with $|P'|<|P|$, by the choice of $P$, the subgraph $G[V(C\cup P')]$ contains a local jump over $C$ across one vertex or a short jump over $C$, so is $G[V(C\cup P)]$, which is a contradiction. This proves Lemma \ref{jump-local}.
\end{proof}

Let $C$ be a cycle and suppose that $x_1,x_2,\ldots,x_n\in V(C)$ occur on $C$ in this cyclic order with $n\geq3$. For any two distinct $x_i$ and $x_j$, the cycle $C$ contains two $(x_i,x_j)$-paths. Let $C(x_i,x_{i+1},\ldots, x_j)$ denote the $(x_i,x_j)$-path in $C$ containing $x_i,x_{i+1},\ldots, x_j$ (and not containing $x_{j+1}$ if $i \neq j+1$), where subscripts are modulo $n$. Such path is uniquely determined as $n \geq 3$.

\begin{lemma}\label{cross short jump}
Let $C$ be an odd hole of a graph $G\in \mathcal{G}_{\ell}$, and $P_i$ be a short $(u_i,v_i)$-jump over $C$ for each integer $1\leq i\leq 2$. If $P_1,P_2$ are crossing, then $G$ contains an odd $K_4$-subdivision or a balanced $K_4$-subdivision of type $(1,2)$.
\end{lemma}
\begin{proof}
Since $P_1,P_2$ are short jumps over $C$, either $P_1C(v_1,u_2)P_2C(v_2,u_1)$ or $P_1C(u_1,u_2)P_2C(v_2,v_1)$ has length at most $4\ell$ by Lemma \ref{easy case}. By symmetry we may assume that the first cycle is shorter than the last one. Set $C':=P_1C(v_1,u_2)P_2C(v_2,u_1)$. Since $g(G)=2\ell+1$, either $C'$ is a hole or $|C'|=4\ell$ and $C'$ has exactly one or two chords. When $C'$ is a hole, implying that $|C(u_1,u_2)|, |C(v_2,v_1)|\geq2$, the subgraph  $C\cup P_1\cup P_2$ is induced and isomorphic to a balanced $K_4$-subdivision of type $(1,2)$. So we may assume that $|C'|=4\ell$ and $C'$ has exactly one or two chords. Since $|C'|=4\ell$ and $g(G)=2\ell+1$, all cycles in $G[V(C')]$ using exactly one chord of $C'$ are odd holes. Hence, when $C'$ has exactly two chords, the two chords of $C'$ are crossing and $G[V(C')]$ is isomorphic to an odd $K_4$-subdivision; and when $C'$ has exactly one chord, $G$ contains a balanced $K_4$-subdivision of type $(1,2)$.
\end{proof}

The proofs of Lemma \ref{uncrossing jumps} and Theorem \ref{final} use some ideas in \cite{MC22}.
\begin{lemma}\label{uncrossing jumps}
Let $\ell\geq4$ be an integer and $C$ be an odd hole of a graph $G\in \mathcal{G}_{\ell}$. For any integer $1\leq i\leq2$, let $P_i$ be a short or local $(u_i,v_i)$-jump over $C$ such that $\{u_1,v_1\}\neq \{u_2, v_2\}$ and $u_1,u_2, v_2,v_1$ appear on $C$ in this order. Assume that $P_i$ is across $C^*(u_i,v_i)$ and if $P_i$ is local, we have $|V(C^*(u_i,v_i))|=1$ for any integer $1\leq i\leq2$. Then the following hold.
\begin{itemize}
\item[(1)] When $P_1,P_2$ are short, $G$ has an odd $K_4$-subdivision.
\item[(2)] At most two vertices in $V(C(u_1,v_1)\cup C(u_2,v_2))$ are not in a jump hole over $C$.
\end{itemize}
\end{lemma}
\begin{proof}
Without loss of generality we may assume that $P_1,P_2$ are chosen with $P_1^*\cup P_2^*$ minimal. Set $H:=P_1\cup P_2\cup C(u_1,u_2)\cup C(v_1,v_2)$. By symmetry we may assume that $|C(u_1,u_2)|\leq |C(v_1,v_2)|\geq1$. Note that $u_1$ maybe equal to $u_2$.

\begin{claim}\label{2short}
(1) is true.
\end{claim}
\begin{proof}[Subproof.]
Since $P_1,P_2$ are short jumps over $C$, by Lemmas \ref{easy case} and \ref{jump parity}, $|P_1\cup P_2|\leq4\ell-2$ and $|H|$ is odd and at most $6\ell-5$. Then $P_1\cup P_2$ contains at most one cycle. By the choice of $P_1$ and $P_2$, the subgraph $P_1\cup P_2$ contains no cycle. So $H$ has at most two cycles. When $H$ has exactly two cycles, since $|H|\leq 6\ell-5$, one cycle in $H$ is a hole, so we can find a new short jump over $C$ to replace one of $P_1,P_2$ and get a contradiction to the choice of $P_1,P_2$. Hence,  $H$ contains a unique cycle $C'$.  Since $g(G)=2\ell+1$, the cycle $C'$ contains at most two chords, and the two chords are uncrossing when $C'$ has two chords. No matter which case happens, $G[V(P_1\cup P_2)]$ contains two short jumps $Q_1,Q_2$ over $C$ such that $C\cup Q_1\cup Q_2$ is isomorphic to an odd  $K_4$-subdivision. So (1) holds.
\end{proof}

By \ref{2short}, we may assume that some $P_i$ is local. For any integer $1\leq i\leq2$, let $a_i, b_i$ be the vertices in $V(P_i)$ adjacent to $u_i,v_i$, respectively, and set $D_i:=P_i^*-\{a_i,b_i\}$. Since $\ell\geq3$, both $D_1$ and $D_2$ are nonempty by Lemma \ref{easy case}. Since $v_1\neq v_2$ and $P_1, P_2$ are jumps over $C$ across $C^*(u_1,v_1)$ and $C^*(u_2,v_2)$ respectively, we have that $b_1\neq b_2$ and they are nonadjacent.

We say two vertex disjoint subgraphs of a graph are {\em anticomplete} if there are no edges between them.
\begin{claim}\label{d1d2}
Either (2) is true or $D_1\cup\{b_1\}$ is disjoint and anticomplete to $D_2\cup\{b_2\}$.
\end{claim}
\begin{proof}[Subproof.]
Assume not. Since $b_1\neq b_2$ and they are nonadjacent, by symmetry we may assume that $G[V(D_1\cup D_2)\cup\{b_1\}]$ is connected. We claim that $P_2$ is short. Assume not. Let $t$ be the vertex in $V(C)$ adjacent to $u_2,v_2$. Since $D_2\neq\emptyset$, there is a minimal path $Q$ linking $V(C(u_1,v_1))-\{u_1\}$ and $t$ with interior in $V(D_1\cup D_2)\cup\{b_1\}$. Let $s$ be the other end of $Q$. Since (2) holds when $s\neq v_1$, we have $s=v_1$. When $Q$ is across $V(C^*(v_1,u_1, u_2,t))$, (2)  holds; when $Q$ is across $V(C^*(v_1,v_2, t))$, replacing $P_2$ by $Q$,  we get a contradiction to the choice of $P_1, P_2$. Hence, $P_2$ is short, implying that $P_1$ is local and $u_1\neq u_2$, for otherwise (2) holds.

Let $Q$ be a minimal $(s,t)$-path linking $V(C(u_1,v_1))-\{u_1\}$ and $\{u_2,v_2\}$ with interior in $V(D_1\cup P^*_2)\cup\{b_1\}$, with $s\in V(C(u_1,v_1))-\{u_1\}$ and $t\in \{u_2,v_2\}$. Since no vertex in $V(C)-\{s,t\}$ has a neighbour in $V(Q^*)$, either $Q$ is a short $(s,t)$-jump over $C$ or $st\in E(C)$. Assume that $st\in E(C)$. Since $1\leq |C(u_1,u_2)|\leq |C(v_1,v_2)|$, we have $s=v_1$, $t=v_2$, and $|C(u_1,u_2)|=1$. Since $P_2$ is short and $P_1$ is across one vertex, by Lemmas \ref{easy case} and \ref{jump parity}, we have $|P_2|=|C(u_2,u_1,v_1,v_2)|=4\geq\ell+1$, so $\ell\leq 3$, which contradicts to the fact that $\ell\geq4$. Hence, $Q$ is a short $(s,t)$-jump over $C$.

When $t=u_2$, since $P_1$ is local and $1\leq|C(u_1,u_2)|\leq |C(v_1,v_2)|$, the short jump $Q$ is across $V(C^*(s,u_1, u_2))$, so (2) holds as $P_1$ is short. When $t=v_2$,  either (2) holds or we can replace $P_1$ by $Q$ to get a contradiction to the choice of $P_1,P_2$. This proves \ref{d1d2} as $t\in \{u_2,v_2\}$.
\end{proof}
By \ref{d1d2}, we may assume that $D_1\cup\{b_1\}$ is disjoint and anticomplete to $D_2\cup\{b_2\}$. By Lemma \ref{jump parity}, $|H|$ is odd and at least $2\ell+3$. By symmetry and \ref{d1d2}, we may assume that $a_1$ has a neighbour in $V(D_2)\cup\{b_2\}$. Let $w$ be the vertex in $N(a_1)\cap (V(D_2)\cup\{b_2\})$ closest to $v_2$ along $P_2$. Set $Q:=u_1a_1wP_2(w,v_2)$. Then $Q$ is a jump of $C$. Since (2) holds when $Q$ is short, we may assume that 
$Q$ is not short, implying that $P_2$ is local and the vertex adjacent to $v_2, u_2$ has a neighbour in $V(P_2(w,v_2))$. Then $|Q|\geq 3\ell-1$, implying that $P_1(a_1,v_1)C(v_1,v_2)Q(v_2,a_1)$ is an even hole by \ref{d1d2}. Hence, by Lemma \ref{jump parity}, $C(u_1,v_1,v_2)Q$ is an odd hole of length at least $3\ell+2$, which is not possible.
\end{proof}

\begin{theorem}\label{final}
Let $\ell\geq5$ be an integer and $C$ be an odd hole of a graph $G\in \mathcal{G}_{\ell}$. Then one of the following holds.
\begin{itemize}
\item[(1)] $G$ has an odd $K_4$-subdivision.
\item[(2)] $G$ contains a balanced $K_4$-subdivision of type $(1,2)$.
\item[(3)] $G$ has a $P_3$-cut.
\item[(4)] $G$ has a degree-$2$ vertex.
\end{itemize}
\end{theorem}
\begin{proof}
Assume that neither (1) nor (2) is true. Set $C:=v_1v_2\ldots v_{2\ell+1}v_1$. Let $P$ be a short jump over $C$  with $|P|$ as small as possible. By symmetry we may assume that the ends of $P$ are $v_2,v_k$ with $k\leq\ell+2$. By Lemmas \ref{uncrossing jumps} (1), \ref{cross short jump} and \ref{easy case}, we may assume that all short jumps over $C$ have ends in $\{v_2,v_3,\ldots,v_k\}$. That is, {\bf (i)} no jump hole over $C$ contains a vertex in $V(C)-\{v_2,v_3,\ldots,v_k\}$.

For each integer $1\leq i\leq 2$, let $P_i$ be a local $(s_i,t_i)$-jump over $C$ across one vertex and $Q_i$ be the $(s_i,t_i)$-path on $C$ of length three. 
When $P_1$ and $P_2$ are uncrossing, by (i) and Lemma \ref{uncrossing jumps} (2), $|V(Q_1\cup Q_2)-\{v_2,v_3,\ldots,v_k\}|\leq2$. When $P_1$ and $P_2$ are crossing, $|Q_1\cup Q_2|=4$. Since there is no short jump over $C$ or at least one vertex in $\{s_i,t_i\}$ is in $\{v_2,v_3,\ldots,v_k\}$ by Lemma \ref{uncrossing jumps} (2), by possibly reordering we may assume that all local jumps over $C$ across one vertex have ends in $\{v_1,v_2,\ldots, v_k, v_{k+1}\}$ with $4\leq k\leq\ell+2$. For any integer $1\leq i\leq k+1$, let $X_i$ be the set of vertices adjacent to $v_i$ that are in a local jump over $C$ across one vertex with one end $v_i$ or a short jump over $C$ with one end $v_i$.  Set $X:=X_1\cup X_2\cup \cdots \cup X_{k+1}$. Then $X$ intersects all short jumps over $C$ and all local jumps over $C$ across one vertex in at least two vertices.

Since $\ell\geq5$ and $k\leq\ell+2$, we have $k+3\leq 2\ell$.
Since no vertex in $V(G)-V(C)$ has two neighbours in $V(C)$, the vertex $v_{k+3}$ has no neighbour in $X$. Assume that $v_{k+3}$ has degree at least three, for otherwise (4) holds. There is a connected induced subgraph $D$ such that $v_{k+3}$ has a neighbour in $V(D)$ and $V(D)\cap(V(C)\cup X)=\emptyset$, and $D$ is maximal with these properties. Let $N$ be the set of vertices in $V(C)\cup X$ that have a neighbour in $V(D)$. Evidently, $v_{k+3}\in N$.

\begin{claim}\label{N}
$N\cap(X\cup V(C))\subseteq\{v_{k+2},v_{k+3},v_{k+4}\}$.
\end{claim}
\begin{proof}[Subproof.]
Assume not. When $1\leq i\leq k+1$, set $W_i:=X_i\cup\{v_i\}$; and when $k+5\leq i\leq 2\ell+1$, set $W_i:=\{v_i\}$. Assume that $N\cap W_i\neq\emptyset$ for some integer $i\notin\{{k+2},{k+3}, {k+4}\}$. Let $Q$ be a minimal $(v_i,v_{k+3})$-path with interior in $V(D)\cup(W_i-\{v_i\})$. Then $Q$ is a jump over $C$ and $|Q\cap X|\leq1$. Since $X$ intersects each local jumps over $C$ across one vertex and each short jump over $C$ in at least two vertices, $G[V(C\cup Q)]$ does not contain a local jump over $C$ across one vertex or a short jump over $C$, which is not possible by Lemmas \ref{jump} and \ref{jump-local}
\end{proof}

By \ref{N}, the set $\{v_{k+2},v_{k+3},v_{k+4}\}$ is a $P_3$-cut of $G$. So (3) holds. 
\end{proof}

Now, we can prove Theorem \ref{main thm}, which is restated here for convenience.

\begin{theorem}\label{}
For any integer $\ell\geq5$, each graph in $\mathcal{G}_{\ell}$ is $3$-colorable.
\end{theorem}

\begin{proof}
Assume not. Let $G\in \mathcal{G}_{\ell}$ be a minimal counterexample to Theorem \ref{main thm}. Then $G$ is 4-vertex-critical. So $G$ has no degree-2 vertex. By Theorems \ref{ex-odd K4} and \ref{bal k4(1,2)},
$G$ contains neither odd $K_4$-subdivision nor balanced $K_4$-subdivision of type $(1,2)$. Hence, $G$ has a $P_3$-cut by Theorem \ref{final}, which is a contradiction to Lemma \ref{P3}.
\end{proof}

\section{Acknowledgments}
This research was partially supported by grants from the National Natural Sciences
Foundation of China (No. 11971111). 
The author thanks Yidong Zhou for carefully reading the paper and giving some suggestions.

\end{document}